\documentclass[a4paper,12pt]{amsart}

\usepackage{color}
\usepackage{amsxtra}
\usepackage{graphicx}
\graphicspath{ {./}}
\usepackage{comment}
\usepackage{bbm} 
\usepackage{amsmath,amssymb,amsthm,graphicx,url}
\usepackage{hyperref}

\newtheorem{thm}{Theorem}
\newtheorem{prop}{Proposition}[section]
\newtheorem{cor}[prop]{Corollary}
\newtheorem{lem}[prop]{Lemma}
\newtheorem{defi}[prop]{Definition}
\newtheorem*{rem}{Remark}

\newcommand{\Q}{\mathbb{Q}}
\newcommand{\eps}{\varepsilon}

\newcommand{\CIc}{{\mathcal C}^\infty_{\rm{c}} }

\newcommand{\C}{\mathbb C}
\newcommand{\N}{\mathbb N}

\newcommand{\R}{\mathbb R}
\newcommand{\Z}{\mathbb Z}

\newcommand{\Op}{\operatorname{Op}}
\newcommand{\Tr}{\operatorname{Tr}}
\newcommand{\supp}{\operatorname{supp}}

\newcommand{\cjg}{\operatorname{\langle}}
\newcommand{\cjd}{\operatorname{\rangle}}
\title{Quantum ergodicity for pseudo-Laplacians}
\author{Elie Studnia}
\begin{document}
\maketitle
\begin{abstract}
We prove quantum ergodicity for the eigenfunctions of the pseudo-Laplacian on surfaces with hyperbolic cusps and ergodic geodesic flows.
\end{abstract}

Quantum ergodicity states that for quantum systems with ergodic classical flow, almost all high-frequency eigenfunctions are equidistributed in phase-space. Quantum unique ergodicity corresponds to equidistribution of all high-frequency eigenfuctions. The main examples are given by 
compact Riemannian manifolds $(X,g)$ with ergodic geodesic flows, where one considers eigenfunctions of the Laplacian $\Delta_g$ associated to the metric, and negatively curved metrics are the typical models for ergodic geodesic flows. 

The first mathematical results in this direction are due to Schnirelman \cite{Sc}, and later by Zelditch \cite{Ze1} and Colin de Verdi\`ere \cite{CdV4}, who proved quantum ergodicity for closed manifolds with ergodic geodesic flows. In the case of manifolds with boundary, similar results were shown by G\'erard-Leichtnam \cite{GeLe} and Zelditch-Zworski \cite{ZeZw}.

In this work, we consider cases of non-compact manifolds, and the first examples one has in mind are surfaces with finite volume. In general, non-compactness often produces an essential spectrum for the Laplacian, and this is indeed the case for the simplest model of finite volume surfaces, namely hyperbolic surfaces realised as quotients $\Gamma\backslash \mathbb{H}^2$ of the 
hyperbolic plane by Fuchsian subgroups with a finite index. In that setting, there is however a way to get rid of this essential spectrum by a simple modification of the Laplacian, that is called \emph{pseudo-Laplacian}, introduced by Colin de Verdi\`ere \cite{CdV1, CdV2} (and related to the work 
Cartier-Hejhal \cite{CaHe}). This operator was very useful for obtaining a meromorphic extension of the Eisenstein series and the resolvent of the Laplacian, with important generalisation to the higher rank 
case by M\"uller \cite{Mu}. There is however a different route in the setting of hyperbolic surfaces with finite volume that was taken by Zelditch \cite{Ze2}, who proved a quantum ergodicity statement for the Laplacian, but it involves the contribution of the continuous spectrum through the Eisenstein series. The proof has been recently generalised by Bonthonneau-Zelditch \cite{BoZe} to variable curvature and all dimensions, while Jakobson \cite{Ja} and Luo-Sarnak \cite{LuSa} proved some quantum unique ergodicity in the case of the modular surface. 
The problem of quantum ergodicity for the eigenfunctions of the pseudo-Laplacian was first proposed by S. Zelditch \cite{Ze-Pr}. This means that in this paper, we deal with the case of the pseudo-Laplacian, that we denote by $\Delta_c$, which is defined as an unbounded operator on $L^2$,with domain $\mathcal{D}_c$ and that has discrete spectrum. Here, $\Op_h$ will denote a semiclassical quantization for compactly supported symbols, see section $1$. Our main result is the following quantum ergodicity statement for this operator:
\begin{thm}\label{main}
Let $X$ be a Riemannian surface with a finite number of constant curvature hyperbolic cusps such that the geodesic flow on $S^* X$ is ergodic.
Let $u_j \in {\mathcal D}_c$ be
an orthonormal family of eigenfuctions of $\Delta_c - \frac{1}{4}$
with eigenvalues $ \lambda_0^2 < \lambda_1^2 \leq \lambda_2^2 
\leq \cdots $, covering all the eigenvalues of $\Delta_c - \frac{1}{4}$ except a finite number of non-positive ones. Let $a \in S^0(T^* X)$ be compactly supported in space.  

Then, as $\lambda\to \infty$, we have 
$$\frac{ 1 } { N( \lambda ) } 
\sum_{ \lambda_j \leq \lambda} \left| 
\langle \Op_{h_j}( a ) u_j , u_j \rangle - 
\int_{ S^* X } a \right|^2 \longrightarrow 0,$$
where $ N ( \lambda ) = | \{ j: \lambda_j \leq \lambda \}|$ and $h_j = \lambda_j^{-1}$. 
\end{thm}
For a precise review of the geometry of the considered Riemannian manifolds, we refer to Section \ref{notations}, while for the definition of the pseudo-Laplacian this is done in Section \ref{pseudo}. 
There are very natural examples of such manifolds given by negatively curved surfaces with finite volume and hyperbolic cusps.

Let us make several remarks about the Theorem. First, by a standard argument (see for instance 
\cite[Section 15.5]{EZ}) Theorem \ref{main} implies that 
\[ \langle \Op ( a ) u_j , 
u_j \rangle \rightarrow \int_{S^* X }a  \] 
for a sequence of density one, when $a$ has compact support. Moreover, since we are only interested in quantizing symbols with a compact support in the space variable, we can use a standard quantization procedure, see for instance \cite[Section 14.2]{EZ}. That means however that the estimates are not uniform far in the cusp.

In the same geometric setting, we also mention that there are other works by Dyatlov \cite{Dy} and Bonthonneau \cite{Bo} on the microlocal limits of 
non-$L^2$ eigenfunctions of the Laplacian but with complex eigenvalues, where one instead get a sort of ``quantum unique ergodicity".

For simplicity, the proof will be presented in the case where there is one cusp, the argument being the same with several cusps. The method of proof follows the scheme from \cite{ZeZw} and has two steps:\\
1) a pointwise ``ellipticity bound" that states that the eigenfunctions are microlocalized on the cosphere bundle. This implies that in the limit $\lambda \rightarrow \infty$, $$M(a,\lambda)^2 := \frac{ 1 } { N( \lambda ) } 
\sum_{ \lambda_j \leq \lambda} \left| 
\langle \Op_{h_j}( a ) u_j , u_j \rangle - 
\int_{ S^* X } a \right|^2$$ is controlled by $\|a_{|S^*X}\|_{L^2}^2$.

\noindent 2) Taking a symbol with average zero, we propagate it by the geodesic flow to get a new symbol that is small on the cosphere bundle (by the $L^2$ ergodic theorem); we have to prove that this does not modify much $M(a,\lambda)$: this is the point of the ``flow invariance" theorems. 

We stress that working with a pseudo-Laplacian entails new difficulties, as compared to the compact setting. 
For the first step, since we are working with a pseudo-Laplacian, the pointwise ellipticity bound (and the subsequent microlocalization) works only outside the singular circle, and we need to prove that the needed correction is small enough. This requires a precise control of the eigenfunctions of the pseudo-Laplacian. \\
For the second step, it is important to notice that the eigenfunctions we are interested in are not eigenfunctions of the propagator we are using for the proof. 
We are able to prove that $M(a,\lambda)$ does not change much when replacing $a$ by $a\circ \Phi^t$ if $\Phi^t$ is the geodesic flow, but we have to assume that the symbol $a$ is supported quite far away from the singular circle. Since the admissible support has full measure, the $L^2$ control of $\limsup M(a,\lambda)$ we still get at step $1$ leads to the same result.\\ 
Finally, \cite{ZeZw} work with a compact manifold, and $M(1,\lambda)$ thus vanishes. 
This is not the case for us since we only use symbols with a compact support in space. 
Our proof has a third step which consists in finding symbols $a$ with average close to $1$ such that $\limsup M(a,\lambda)$ is arbitrarily close to zero. For that purpose, we shall prove that the modes of the eigenfunctions of the pseudo-Laplacian are microlocalized in the cusp.  

\textbf{Acknowledgements}
This work was written during a visit at UC Berkeley, under the direction of S. Dyatlov and M. Zworski. We thank this institution and S. Dyatlov and M. Zworski for their help, suggestions and comments. We also thank C. Guillarmou for helpful comments on a first version of the paper. Partial support form the National Science Foundation grant DMS-1500852 is also acknowledged.

\section{Preliminaries}

\subsection{Notations}\label{notations}

We let $X$ be a Riemannian surface with one hyperbolic cusp, i.e. a cusp with constant curvature. This means that $X$ can be split into two parts, $X = X_0 \cup X_1$, where $X_1$ is a compact Riemannian surface with boundary, and $X_0 = (c_0,\infty)_r \times (\R/\Z)_\theta$ with metric $dr^2+e^{-2r}d\theta^2$. Using the notation $\xi dr+\eta d\theta$ for cotangent vectors in $X_0$, 
the Hamiltonian induced by the metric in the cusp is given by 
\[  p(r,\theta;\xi,\eta) = \xi^2 + e^{ 2 r } \eta^2 . \]
In $X_0$, any $u \in L^2_{\text{loc}}$ function can be expanded into Fourier series in the $\theta$ variable:
\[u(r,\theta) = \sum_{n \in \Z}{u_n(r)}e^{2i\pi n\theta},\]
where the $u_n$ are in $L^2_{{\rm loc}}((c_0,\infty);e^{-r}dr)$. 
The metric induces a natural measure $\mu$, called Liouville measure, on the unit cotangent bundle $S^*X$ and for simplicity we shall normalize it so that it is a probability measure. The projection $T^*X\to X$ on the base will be denoted by $\pi$.
Finally, $C>0$ will denote a generic constant that is independent of the parameters we consider (except when indicated), and that will change from line to line.

\subsection{Definition of the pseudo-Laplacian}\label{pseudo}

\begin{defi}
Let $c > c_0$. Let us denote $L^2_{0,c}$ (resp. $H^1_{0,c}$) the space of all $u \in L^2(X)$ (resp. $u \in H^1(X)$) such that $u_0(r)=0$ for every $r \geq c$. 
The pseudo-Laplacian $\Delta_c$ is the unbounded non-negative self-adjoint operator on $L^2_{0,c}$ defined by the quadratic form using the Friedrichs method 
\[\forall u\in H^1_{0,c}, \quad  q(u) = \int_X{|\nabla^g u|_g^2}d{\rm v}_g.\] 
The Riemannian measure $d{\rm v}_g$ and the gradient are with respect to $g$.
\end{defi}
We note that the spaces $L^2_{0,c}$ and $H^1_{0,c}$ are closed vector subspaces of $L^2(X)$, and $H^1(X)$. The circle $r=c$ in $X_0$ will be referred to as \emph{the singular circle}.

The following results are proved in \cite[Theorem 2]{CdV3}. 

\begin{prop}
The operator $\Delta_c$ is an unbounded, non-negative, self-adjoint operator with compact resolvent and discrete spectrum.
\end{prop}

We will denote $(u_j)_j$ an orthonormal family of eigenfunctions with positive eigenvalues of $\Delta_c-\frac{1}{4}$, that is, $\Delta_cu_j = \left(\lambda_j^2+\frac{1}{4}\right)u_j$, where $(\lambda_j)_j$ is a positive, non-decreasing sequence going to $+\infty$. Note that the orthogonal of 
${\rm Span}\{u_j,\,j \geq 0\}$ in $L^2_{0,c}$ is a finite-dimensional space that possesses an orthonormal basis of eigenfuctions of $\Delta_c$. 
We will denote, for each $j \geq 0$, $h_j = \lambda_j^{-1}$.

Note that we extend $\Delta_c$ as an unbounded self-adjoint operator from $L^2$ to $L^2$ with compact resolvent by declaring that $\Delta_cv=0$ whenever $v \in L^2(X)$ has support in $\{r \geq c\}$ and $v$ does only depend on $r$. 

\subsection{Review of semiclassical analysis}
We shall use the following semiclassical quantization procedure, which is similar to \cite[Chapter 14.2]{EZ}: we fix a cover by countably may open sets $U_i$ of $X$, $i \geq 0$, 
with diffeomorphisms $\varphi_i:\,U_i \rightarrow V_i$, where the $V_i \subset \R^2$ are open sets, and take a partition of unity $(\chi_i^2)_i$ associated with it. A compactly supported symbol $a\in S_{\rm comp}^m(X)$ is a smooth function $a\in C^\infty(T^*X)$ whose support projects to $X$ into a compact set 
 and satisfying uniform bounds 
 \[ |\partial_x^\alpha \partial_\xi^\beta a(x,\xi)|\leq C_{\alpha,\beta}\cjg \xi\cjd^{m-|\beta|}
  \]
for all multi-indices $\alpha,\beta$.
Then for any symbol $a \in S_{\rm comp}^m(X)$ with compact space support, and $h > 0$, we define 
\[\Op_h(a):= \sum_i{\chi_i(\varphi_i)^*((\varphi_i)_*a)^w(x,hD)(\varphi_i)_*\chi_i}.\]
where $b^{w}(x,hD)$ means the Weyl quantization.
When $U_i\cap \pi(\supp(a))=\emptyset$ (which always happen but for a finite number of $i$), $i$ does not contribute to the sum, because $((\varphi_i)_*a) = 0$. In any case, $(\varphi_i)_*a \in S^m_{{\rm loc}}(V_i)$. 

The specific choice of the partition of unity is not important, because the difference between two different such quantizations is then an $\mathcal{O}(h)_{L^2 \rightarrow L^2}$ for any $S_{\rm comp}^0(X)$ symbol. We shall thus make the following choices:
\begin{itemize}
\item $V_0 = V_1 = (c-\eps,c+\eps) \times (0,1)$
\item $U_0 = (c-\eps,c+\eps) \times (0,1) \subset (c_0,\infty) \times \R/\Z$ and $\varphi_0$ is the Identity. 
\item $U_1 = (c-\eps,c+\eps) \times \left(-\frac{1}{2},+\frac{1}{2}\right) \subset (c_0,\infty) \times \R/\Z$ and $\varphi_1$ is a shift in the second coordinate only. 
\item $\chi_0$ and $\chi_1$ are only $r$-dependent. 
\item For every $i \geq 2$, $\chi_i = 0$ in $\left\{c-\frac{\eps}{2} < r < c+\frac{\eps}{2}\right\}$. 
\item Every $V_i$ is convex. 

\end{itemize}

With this procedure all the useful properties (about composition, Lie brackets, $L^2$ operators bounds for $S^0$ quantized symbols) hold: the proofs from \cite[Chapters 14, 15]{EZ} still apply when the symbols are compactly supported in $x$, however the constants depend on the size of the supports.

\section{Estimates on the singular circle}

\subsection{Riemannian Laplacian of the eigenfunctions}
In this section, we study the the family $(\Delta-\lambda_j^2-\frac{1}{4})u_j$.
We will denote by $\delta_c$ the Lebesgue measure with total mass $1$ 
on the circle $r=c$ of the cusp of $X$. 

The following lemma is an easy application of Stokes's theorem. 
\begin{lem}
\label{laplacian_truncated}
Let $\varphi$ be a smooth function on $X$ such that $\varphi(r,\theta)=\tilde{\varphi}(r)$ in the cusp on $r > c-\eps$ for some $0 < \eps < c-c_0$, where $\tilde{\varphi}:(c_0, \infty) \rightarrow \C$ is smooth. Assume that $\tilde{\varphi}(c) = 0$, then $\Delta(\mathbbm{1}_{r \geq c}\varphi) = \mathbbm{1}_{r \geq c}\Delta\varphi-e^{-c}\tilde{\varphi}'(c)\delta_c$.
\end{lem}
Now, we can write the Laplacian of $u_j$ as a function of its zero mode.  

\begin{cor}
\label{defi_alphaq}
For $j \geq 0$ and $h_j=\lambda_j^{-1}$, we have $(u_j)_0(r) = \alpha_je^{r/2}\sin{\frac{r-c}{h_j}}$ for some $\alpha_j\in\mathbb{R}$ in the region $c_0 < r \leq c$, it vanishes when $r \geq c$, and 
\[\left(\Delta-\lambda_j^2-\frac{1}{4}\right)u_j = Q_j\delta_c,\] 
where $Q_j=+e^{-c/2}\frac{\alpha_j}{h_j}$.
\end{cor}

\begin{proof} 
On $\{c_0 < r < c\}$, $(u_j)$ is an eigenfunction of the positive Riemannian Laplacian $\Delta$ with eigenvalue $\lambda_j^2+\frac{1}{4}$. In our coordinates, $\Delta = -\partial_r^2 + \partial_r - e^{2r}\partial_{\theta}^2$. On the zero Fourier mode of $u_j$, $\partial_{\theta}$ acts as $0$, thus $(-\partial_r^2+\partial_r)(u_j)_0 = (\lambda_j+\frac{1}{4})^2(u_j)_0$. Setting $(u_j)_0(r) = v_0(r)e^{r/2}$ yields $-\partial_r^2v_0 = \lambda_j^2v_0$. The boundary condition $(u_j)_0(c) = 0$ then gives the formula for $(u_j)_0$. \\
From the proof of Theorem $4$ in \cite{CdV3}, $$u'_j := u_j + \mathbbm{1}_{r \geq c}\alpha_je^{r/2}\sin{\frac{r-c}{h_j}}$$ is a non-$L^2$ eigenfunction of the positive Laplacian with eigenvalue $\lambda_j^2 + \frac{1}{4}$. Therefore, using lemma \ref{laplacian_truncated}, 
\begin{align*}
(\Delta-\Delta_c)u_j &= \left(\Delta-\lambda_j^2-\frac{1}{4}\right)u_j = \left(\Delta - \lambda_j^2-\frac{1}{4}\right)(u_j-u'_j) \\
&=- \alpha_j\left(\Delta-\lambda_j^2-\frac{1}{4}\right)\left(\mathbbm{1}_{r \geq c}e^{r/2}\sin{\frac{r-c}{h_j}}\right) \\
&= +e^{-c/2}\frac{\alpha_j}{h_j}\delta_c.
\end{align*}
This completes the proof.
\end{proof}

To estimate the $\Delta u_j$, we need an adequate description of the constants $\alpha_j$ from Corollary \ref{defi_alphaq}.  

\begin{prop}
There exists a smooth compactly supported function $\tilde{\phi}$ on $X$, and a sequence $(I_j)_{j \geq 0}$ such that $I_j\alpha_j = \langle u_j,\,\tilde{\phi}\rangle$ (it is the $L^2$ inner product) for every $j \geq 0$ and $I_j = -h_j+\mathcal{O}(h_j^{\infty})$. 
\end{prop}

\begin{proof}
Let $\phi$ be any smooth compactly supported function on $\R$ such that:
\begin{itemize}
\item $\phi=0$ on $\left(-\infty,\frac{c_0+c}{2}\right)$
\item $\phi(c) = 1$
\item for every $p \geq 1$, $\phi^{(2p)}(c) = 0$.
\end{itemize}
Let $\tilde{\phi}(r,\theta) = e^{r/2}\phi(r)$ (and $\tilde{\phi}$ is zero outside the cusp), such that $\tilde{\phi}$ is well-defined on $X$, smooth, compactly supported. 
Now, since $\tilde{\phi}$ has no non-zero $\theta$-Fourier mode, using its 
support property and the nature of the hyperbolic metric, we know that 
$\langle u_j,\,\tilde{\phi}\rangle = \alpha_jI_j$, where 
\[I_j = \int_{(c_0+c)/2}^c{\phi(r)\sin{\frac{r-c}{h_j}}dr} = \int_{-\infty}^c{\phi(r)\sin{\frac{r-c}{h_j}}dr} = -h_j+h_j\int_{-\infty}^c{\phi'(r)\cos{\frac{r-c}{h_j}}}.\]

Set $\phi_1(r) = \phi(r+c)$: there exists $\phi_2\in C_c^\infty(\R;\R)$  such that $\phi_2(r)=\phi_1'(r)$ if $r \leq 0$ and $\phi_2(r) = \phi_1'(-r)$ if $r \geq 0$ (recall that all derivatives of odd order of $\phi_1'$ vanish at $0$). Then 

$$2(I_j+h_j) = h_j\int_{-\infty}^0{\phi_2(r)e^{ir/h_j}dr}+h_j\int_0^{\infty}{\phi_2(r)e^{ir/h_j}dr} = h_j\left(\mathcal{F}\phi_2\right)\left(\frac{1}{h_j}\right) = \mathcal{O}(h_j^{\infty})$$ 
and we are done.
\end{proof}

\begin{cor}
\label{ef_are_ell2}
We have: $$\sum_{j}{|Q_j|^2h_j^4} < \infty.$$
\end{cor}

\begin{proof}
Since $h_j \sim -I_j$ as $j$ goes to infinity, and since $Q_j = F\alpha_j(h_j)^{-1}$ for some constant $F$, we find that $|Q_j|^2h_j^4$ is positive and is equivalent to $|F|^2\alpha_j^2I_j^2 = |F|^2\langle u_j,\,\tilde{\phi}\rangle^2$ (with the above notations). Now, since the $(u_j)_j$ forms an orthonormal family in $L^2$, $$\sum_j{\langle u_j,\,\tilde{\phi}\rangle^2} \leq \|\tilde{\phi}\|^2_{L^2} < \infty,$$ which proves the claim.   
\end{proof}

\subsection{Pseudo-differential operators acting on $\delta_c$}

Following up on the previous subsection, we have:

\begin{prop}
\label{singular_S-2}
Let $a \in S_{\rm comp}^{-2}(X)$ with  $\pi(\supp(a))\subset \{c-\eps < r < c+\eps\}$. Then, for some $C > 0$ not depending on $a$, for every $1 > h > 0$, $$\|\Op_h(a)\delta_c\|_{L^2}^2 \leq \frac{C}{h}\int_{\R / \Z}{\int_{\R}{|a|^2(c,\theta,\xi,0)d\xi}d\theta}+ C\|a\|^2_{S^{-2}},$$ where $\|\cdot\|_{S^{-2}}$ is some $S^{-2}$ seminorm (in every estimate of that kind in the following, the seminorm will have to be universal). 
\end{prop}

\begin{lem}
\label{chart_singular}
Let $a\in S_{\rm comp}^{-2}(\R^2)$, with  $\pi(\supp(a))\subset \left(c-\eps,c+\eps\right) \times (0,1)$. Let $\chi:\R \rightarrow [0,1]$ be smooth and zero outside $(0,1)$. Let $\langle\nu,\,\varphi\rangle = \int_0^1{\chi(\theta)\varphi(c,\theta)d\theta}$. 
Then, for some universal constant $C > 0$,
$$\|a^w(x,hD)\nu\|^2_{L^2} \leq \frac{C}{h}\int_{\substack{0 < \theta < 1 \\ \xi \in \R}}{|a(c,\theta,\xi,0)|^2|\chi(\theta)|^2\,d\theta d\xi} + C\|a\|^2_{S^{-2}}$$
\end{lem}

\begin{proof}
We may assume that $a$ is compactly supported, if we find out that $C$ does not depend on the support of $a$.
A computation gives 
\begin{align*}
(2h\pi)^4\|a^w(x,hD)\nu\|^2_{L^2} &= \int_0^1\int_{\R}\int_{\R^6}\chi(\phi)\chi(\phi')a\left(\frac{r+b}{2},\frac{\theta+\phi}{2},\xi,\eta\right)\overline{a}\left(\frac{r+b}{2},\frac{\theta+\phi'}{2},\xi',\eta'\right)\\
&\times\exp\left[\frac{i}{h}\left((r-c)(\xi-\xi') + \theta(\eta-\eta')+(\phi'\eta'-\phi\eta)\right)\right]dr d\phi d\phi' d\eta d\eta' d\xi'd\xi d\theta\\
&:= \int_0^1{\int_{\R}{I(h,\xi,\theta)d\xi}d\theta}
\end{align*}

Let $\varphi_{\xi,\theta}(r,\xi',\phi,\phi',\eta,\eta') = (r-c)(\xi-\xi') + \theta(\eta-\eta')+(\phi'\eta'-\phi\eta)$. $\varphi_{\xi,\theta}$ is a smooth function from $\R^6$ to $\R$, and its gradient is zero at the only point $r=c,\phi=\phi'=\theta,\eta=\eta'=0,\xi'=\xi$. Besides, at that point, the Hessian matrix of $\varphi_{\xi,\theta}$ is
$$\left(\begin{matrix}
0 & -1 & 0 & 0 & 0 & 0\\
-1 & 0 & 0 & 0 & 0 & 0\\
0 & 0 & 0 & 0 & -1 & 0\\
0 & 0 & 0 & 0 & 0 & 1\\
0 & 0 & -1 & 0 & 0 & 0\\
0 & 0 & 0 & 1 & 0 & 0
\end{matrix}\right),$$ so it has full rank and we see from the stationary phase method (say, \cite[Theorem 3.16]{EZ}), that for some constants $F,C$, $$|I(h,\xi,\theta) - Fh^3|a(c,\theta,\xi,0)\chi(\theta)|^2| \leq Ch^4\frac{\|a\|_{S^{-2}}^2}{1+|\xi|^2}.$$ 
The conclusion is easily drawn from this. 
\end{proof}

Now let us prove proposition \ref{singular_S-2}:
\begin{proof}
Let $\psi$ be a smooth function on $X$ such that $\psi=1$ on $\{2|r-c|<\eps\}$, and $\psi=0$ on $\{|r-c|>\eps\}$. Then write $a = a(1-\psi)+a\psi$. The support $\pi(\supp(a(1-\psi)))$ is at distance at most $\eps$ and at least $\frac{\eps}{2}$ from $\{r=c\}$. Therefore, $\|\Op_h(a(1-\psi))\delta_c\|^2_{L^2} \leq C\|a\|^2_{S^{-2}}$, for some universal constant $C > 0$. Now, apply lemma \ref{chart_singular} to the explicit quantization (as explained in section $1.2$) of $a(1-\psi)$ (where the only non-vanishing terms are for the charts $0$ and $1$). 
\end{proof}

\section{Ellipticity and variance bound}

In this section, we complete what we have called in the introduction the first step. We use the results of the former section, as well as an ellipticity estimate similar to the one from \cite{ZeZw}, to prove that the microlocalization of the eigenfunctions on the energy surface still holds, albeit on average only. 

\begin{defi}\label{def1}
We define, for any symbol $a\in S_{\rm comp}^0(X)$ and for any $h > 0$, $\lambda > 0$, 
\begin{align*}
N(\lambda) &:= |\{j,\,\lambda_j \leq \lambda\}|,\\
Y(a,h) &:= h\sqrt{\sum_{h/2 \leq h_j \leq 2h}{\|\Op_{h_j}(a)u_j\|_{L^2}^2}},\\
M(a,\lambda) &:= \sqrt{\frac{1}{N(\lambda)}\sum_{j,\,\lambda_j \leq \lambda}{\left|\langle\Op_{h_j}(a)u_j,\,u_j\rangle-\int_{S^*X}{a}\right|^2}}. 
\end{align*}
\end{defi}

\begin{rem}
The bound $M(a+b,t) \leq M(a,t)+M(b,t)$ holds, and similarly for $Y$.
\end{rem}
Let us mention the following very important result: 

\begin{prop}[Weyl law for Pseudo-Laplacians]\label{Weyl}
There is a constant $C>0$ such that  $N(\lambda) \sim C\lambda^2$ as $\lambda\to \infty$. 
As a consequence, there is $C_1>0,C_2>0$ such that for all $h>0$ small
\[C_1h^{-2}\leq |\{j\in\N;\,h/2 \leq h_j \leq 2h\}| \leq C_2h^{-2}.\] 
\end{prop}
\begin{proof} 
Actually, $N(\lambda)$ is, up to some additive constant, the number of eigenvalues of $\Delta_c$ that are not greater than $\lambda^2 + \frac{1}{4}$. The result is then proved in 
\cite[Theorem 6]{CdV3}. 
\end{proof}

\subsection{Ellipticity in the mean}

\begin{lem}
Let $a \in S_{\rm comp}^0(X)$ be a symbol and assume that $a_{|S^*X} = 0$. Then 
\begin{equation}\label{boundsell}
\|\Op_{h_j}(a)u_j\|^2_{L^2} \leq C|Q_j|^2h_j^3 \int_{\R / \Z}{\int_{\R}{\frac{|a(c,\theta,\xi,0)|^2}{(|\xi|^2-1)^2}d\xi}d\theta} + \mathcal{O}(h_j^2),
\end{equation} 
where the constant in the $\mathcal{O}(h_j^2)$ depends only on some $S^0$ seminorms of $a$ and on $\pi(\supp(a))$, $C$ is universal and $Q_j$ is the constant of Corollary \ref{defi_alphaq}.
\end{lem}

\begin{proof}
Write $a(x,\xi) = b(x,\xi)(|\xi|^2-1)$, where $a$ and $b$ have same (compact) space support and $b$ is $S^{-2}$. Then $\Op_{h_j}(a) = \Op_{h_j}(b)(h_j^2\Delta-1) + \mathcal{O}(h)$, the $\mathcal{O}$ referring to $L^2 \rightarrow L^2$ operator norm, and the constant satisfies the relevant dependencies. \\
Thus, $\Op_{h_j}(a)u_j = Q_jh_j^2\Op_{h_j}(b)\delta_c + \mathcal{O}_{L^2}(h_j)$. Let $\psi$ be a smooth function from $X$ to $[0,1]$ that is $1$ everywhere, except on $\{c-\eps < r < c+\eps\}$, and that is zero on $\{2|r-c| < \eps\}$. Therefore, we may write $\Op_{h_j}(a)u_j = Q_jh_j^2\Op_{h_j}(b\psi)\delta_c + Q_jh_j^2\Op_{h_j}(b(1-\psi))\delta_c + \mathcal{O}(h_j)$. \\
Now, since the phase space support of $b$ does not meet the wave front set of $\delta_c$ (which is $\{r=c,\eta = 0\}$), $\Op_{h_j}(b\psi)\delta_c$ is a smooth $\mathcal{O}(h_j^{\infty})$ function (with the required dependencies for the constants). Besides, proposition \ref{singular_S-2} gives us the upper bound for $\|\Op_h(b(1-\psi))\delta_c\|^2_{L^2}$. 
\end{proof}

\begin{prop}[Ellipticity in the mean]
\label{mean-ell-est}
Let $a_j \in S_{\rm comp}^0$ for every $j \geq 1$ and assume that $\cup_j(\pi(\supp(a_j))\subset \mathcal{K}$ for some fixed compact set $\mathcal{K}\subset X$ and that the family is bounded in $S^0$. Assume that for each $j$, $\left(a_j\right)_{|S^* X} = 0$. Then $$h^2\sum_{h/2 \leq h_j \leq 2h}{\|\Op_{h_j}(a_j)u_j\|^2_{L^2}} \leq Ch\sup\{\|a_j\|^2_{S^0}\} + \mathcal{O}(h^2),$$ where $\|\cdot\|_{S^0}$ is some $S^0$ seminorm, and $C$ is universal.
\end{prop}

\begin{proof}
Let $I$ be the supremum over $j \geq 0$ of the $$C\int_{\R / \Z}{\int_{\R}{\frac{|a_j(c,\theta,\xi,0)|^2}{(|\xi|^2-1)^2}d\xi}d\theta},$$ where $C$ is the constant in \eqref{boundsell}  and $K$ is the constant in the $\mathcal{O}(h_j^2)$ of \eqref{boundsell}. Then, using Weyl's law and corollary \ref{ef_are_ell2}:
$$h^2\sum_{h/2 \leq h_j \leq 2h}{\|\Op_{h_j}(a_j)u_j\|^2_{L^2}} \leq h\sum_{h/2 \leq h_j \leq 2h}{2I|Q_j|^2h_j^4 + 2Kh^3} \leq 2hC'I + K'h^2,$$ and we conclude by considering that for some suitable $\|\cdot\|_{S^0}$, $I \leq \sup\{\|a_j\|_{S^0}^2\}$. 
\end{proof}

\subsection{Bound on the variance}
In this section, we shall prove some bounds on the variance $M(a,\lambda)$ defined in Definition \ref{def1}.
\begin{lem}
Let $b \in C^\infty_c(T^* X)$. Then, there is some universal constant $C$ such that for all $h > 0$ small 
\[h^2\sum_j{\|\Op_h(b)u_j\|_{L^2}^2} \leq C\int_{T^*X}{|b|^2} + \mathcal{O}(h).\]
\end{lem}
\begin{proof}
We write 
\[h^2\sum_j{\|\Op_h(b)u_j\|_{L^2}^2} \leq h^2\|\Op_h(b)\|^2_{\rm{HS}} = h^2\Tr\left(\Op_h(b)^*\Op_h(b)\right) = h^2\Tr\left(\Op_h(|b|^2+f_h)\right),\] 
where $f_h$ are smooth functions such that $\supp\,f_h \subset \supp\,b$ is compact, and $\|f_h\|_{S^p} = \mathcal{O}(h)$ for each $p \in \Z$. Here ${\rm HS}$ means the Hilbert-Schmidt norm. We finally apply the trace formula in \ref{trace} (in the appendix).
\end{proof}

\begin{prop}\label{boundonY}
Let $a \in S_{\rm comp}^0(X)$. There is a universal constant $C$ such that for all $h > 0$ small, 
\[Y(a,h)^2 \leq C\int_{S^*M}{|a|^2} + \mathcal{O}(h).\]
\end{prop}

\begin{proof}
Let us denote $f_{\tau}(x,\xi) = f(x,\tau\xi)$ for any $\tau > 0$ and any symbol $f$. 
Let $b \in \CIc(T^*X)$ be such that $b(x,\xi) = \chi(|\xi|)a\left(x,\frac{\xi}{|\xi|}\right)$, where $\chi\in C_c^\infty((0,\infty))$ is real valued such that $\chi=1$ on $[2/5,5/2]$. 
From Proposition \ref{mean-ell-est} (with the symbols $a-b$), $Y(a-b,h)=\mathcal{O}(h^{1/2})$, so $Y(a,h)^2 \leq 2Y(b,h)^2+\mathcal{O}(h^{1/2})^2 = 2Y(b,h)^2+\mathcal{O}(h)$. So we may focus on $Y(b,h)$. 

Now, since $\Op_{h\tau}(f) = \Op_h(f_{\tau})$ (because of the quantization procedure), let us denote $\tau_j = h_j/h$. Then, \begin{align*}Y(b,h)^2 &= h^2\sum_{h/2 \leq h_j \leq 2h}{\|\Op_h(b_{\tau_j})u_j\|^2_{L^2}} \\&\leq 2h^2\sum_{h/2 \leq h_j \leq 2h}{\|\Op_{h_j}(b-b_{1/\tau_j})\|^2_{L^2 \rightarrow L^2}} + 2h^2\sum_{h/2 \leq h_j \leq 2h}{\|\Op_h(b)u_j\|_{L^2}^2}.\end{align*}
Since $\int_{T^*X}{|b|^2} \leq C\int_{S^*X}{|a|^2},$
by the previous lemma, it is enough to prove that 
\[\left\|h^2\sum_{h/2 \leq h_j \leq 2h}\Op_{h_j}(b-b_{1/\tau_j})u_j\right\|_{L^2\rightarrow L^2}^2 = \mathcal{O}(h).\] 
Now, since $1/2 \leq \tau_j \leq 2$, the $b-b_{1/\tau_j}$ are bounded in $S^0$, have a uniform support in space $\pi(\supp(b-b_{1/\tau_j}))\subset \pi(\supp(a))$ support, the results follows from Proposition \ref{mean-ell-est}.
\end{proof}

\begin{prop}[Variance bound]\label{variancebound}
If $a \in S^0_{\rm comp}(X)$, then for some universal constant $C$, 
\[M(a,\lambda)^2 \leq C\int_{S^*X}{|a|^2}+\mathcal{O}(\lambda^{-1}).\]
\end{prop}
\begin{proof}
Let $h = 1/(2\lambda)$. Let $I :=\int_{S^*X}{|a|^2}$, where $C$ is as above. Then using Cauchy Schwartz for the $\int a$ term and Proposition \ref{boundonY}, we get
\begin{align*}
N(\lambda)M(a,\lambda)^2 &\leq 2N(\lambda)I+ 2\sum_{1 \leq 4^k \leq 4h_0\lambda}{16^{-k}h^{-2}Y(a,4^kh)^2}\\
&\leq 2N(\lambda)I+ 8C\lambda^2\sum_{1 \leq 4^k \leq 4h_0\lambda}{16^{-k}I+4^{-k}h} \leq C'\lambda^2I+\mathcal{O}(\lambda),
\end{align*}
where the sum is over the $k \geq 0$ such that, for some integer $j \geq 0$, $2^{2k-1}h \leq h_j \leq 2^{2k+1}h$: in particular, if $2^{2k-1}h > h_0$ (ie $4^k \leq 4h_0\lambda$), the corresponding term does not contribute.

We conclude using again Proposition \ref{Weyl}. 
\end{proof} 

\section{Egorov theorem}

This section deals with step $2$: similarly to \cite{ZeZw}, we want to prove that propagating some symbol $a$ through the geodesic flow does not change $M(a,\lambda)$ too much. The main difference here is the fact that the operator we study (ie the pseudo-Laplacian) is not the generator of the propagator we use. From a geometric point of view, we solve this by requiring that our symbols have a support far from the singular circle.  

The main result of this section is proposition \ref{egorov-goal}, which gives a precise statement about the idea above.

\subsection{A good set for propagation}
We define
\[\Sigma_T=\{ ( z, \zeta ) \in T^* X : \ \forall \, |t| \leq T , \Phi^t (z , \zeta ) \notin 
T^*X_0\cap \{\eta = 0\}\} ,\]
where $\Phi^t$ is the geodesic flow on $T^*X$ and the region where the coordinate $\eta$ is defined in the cusp.

\begin{prop}
$ \Sigma_T $ is an open set of full measure.
\end{prop}

\begin{proof}
Let $(z,\zeta) \in \Sigma_T$. Then the set $\{\Phi^t(z,\zeta),\,|t| \leq T\}$ is at positive distance from $\{\eta = 0\}$, say $\eps > 0$. Since the $\Phi^t$, $|t| \leq T$ are equicontinuous (the proof is the same as Heine's theorem), there exists $\eps' > 0$ such that if $(z',\zeta')$ is at distance at most $\eps'$ from $(z,\zeta)$, for every $|t| \leq T$, $\Phi^t(z,\zeta)$ and $\Phi^t(z',\zeta')$ are at distance at most $\eps/2$, thus, $\Phi^t(z',\zeta')$ is at distance at least $\eps/2 > 0$ from $\{\eta = 0\}$. Therefore, $(z',\zeta') \in \Sigma_T$. \\
It remains to prove the "full measure" part: one easily notes that 
\[\Sigma_T = \bigcap_{|r| \leq T, r \in \Q}{\Phi^{-r}(T^* X \backslash (T^*X_0\cap\{\eta = 0\}))}.\] 
Now, $T^*X_0\cap \{\eta = 0\}$ has null measure, thus its complement has full measure, and so has $\Phi^r(T^* X \backslash (T^*X_0\cap\{\eta = 0\}))$ because $\Phi^r$ is a diffeomorphism. 
\end{proof}

\subsection{Flow invariance of the eigenfunctions}

\begin{lem}
Let $T > 0$, let $a \in C_c^\infty(T^* X)$ with  $\supp(a)\subset \Sigma_T$. Then there exists some constant $C > 0$, depending only on $a$ and $T$, such that for every $j \geq 0$, and every $0 \leq t \leq T$, 
$$\left|\langle \Op_{h_j}(a)u_j,\,u_j\rangle-\langle\Op_{h_j}(a \circ \Phi^t)u_j,\,u_j\rangle\right| \leq Ch_j.$$
\end{lem}

\begin{proof}
Let $P:=(h_j^2\Delta -1)$, then we have 
 $Pu_j = h_j^2Q_j\delta_c + \frac{h_j^2}{4}u_j$, and $\delta_c$ is $H^{-1}$ (see the beginning of section $2.1$ for the definition) and $(h_j^2Q_j)_j$ is in $\ell^2(\N)$ (Proposition \ref{ef_are_ell2}) hence bounded. 
Let $s_j(t) := \langle \Op_{h_j}(a\circ \Phi^t)u_j,\,u_j\rangle$; every $s_j$ is smooth, and 
\begin{align*}
\partial_t s_j(t) &= \langle \Op_{h_j}(\{p,a \circ \Phi^t\})u_j,u_j\rangle = -ih_j^{-1}\langle [P,\Op_{h_j}(a\circ\Phi^t)]u_j,\,u_j\rangle +ih_j\langle R(h_j,t)u_j,\,u_j\rangle\\
&= \frac{-i}{h_j}\langle P\Op_{h_j}(a \circ \Phi^t)u_j,\,u_j\rangle - \frac{-i}{h_j}\langle \Op_{h_j}(a \circ \Phi^t)Pu_j,\,u_j\rangle + h_j\langle R(h_j,t)u_j,\,u_j\rangle\\
&= U-V+W.
\end{align*}
In this computation, for every $0 \leq t \leq T$, $R(h,t) \in \Psi^{-\infty}_h(X)$ is a pseudodifferential operator satisfying 
\[\sup\{\|R(h,t)\|_{L^2 \rightarrow L^2};\,0 < h < h_{0},0 \leq t \leq T\} < \infty.\] 
Therefore $|W| \leq C_1h_j$. Next, we get
 \[|V| = \left|\frac{1}{h_j}\langle \Op_{h_j}(a\circ\Phi^t)Pu_j,\,u_j\rangle\right| \leq |Q_j|h_j\left|\langle \Op_{h_j}(a\circ\Phi^t)\delta_c,u_j\rangle\right| + \frac{h_j}{4}\left|\langle\Op_{h_j}(a)u_j,\,u_j\rangle\right|.\] 
 Let us notice that ${\rm WF}_h(\delta_c) \subset \{r=c, \eta = 0\}$, while ${\rm WF}_h(\Op_{h}(a\circ\Phi^t)) \subset \Phi^{-t}(\Sigma_T)$ is at positive distance from $\{\eta = 0\}$ (uniformly in $|t| \leq T$). 
 Therefore, $\|\Op_{h_j}(a \circ \Phi^t)\delta_c\|_{L^2} \leq C'_j(t)h_j^2$, where $(h_j^2Q_jC'_j(t))_{j \geq 1,0 \leq t \leq T}$ is bounded by some $B > 0$. Thus for some constant $B$ depending only on $a$ and $T$, $|V| \leq Bh_j$. 
Finally, 
\[|U| = \left|\frac{1}{h_j}\langle P\Op_{h_j}(a\circ\Phi^t)u_j,\,u_j\rangle\right| = h_j^{-1}\left|\langle \Op_{h_j}(a\circ\Phi^t)u_j,\,Pu_j\rangle\right|,\] 
so let us write $\Op_h(a\circ\Phi^t)^* = \Op_h((\overline{a+hc_{h,t}})\circ \Phi^t)$, where $\bigcup\limits_{h}\supp\,c_{h,t} \subset \supp\,a$. Thus $$ih_jU = \langle\Op_{h_j}(\overline{a+hc_{h,t}}\circ\Phi^t)\delta_c,\,u_j\rangle + \frac{h_j^2}{4}\langle\Op_{h_j}(a)u_j,\,u_j\rangle.$$ The same argument as for the bound on $V$ can be re-used to get $|ih_jU| \leq B'h_j^2$, ie $|U| \leq B'h_j$.  
So we have $|\partial_t s_j(t)| \leq (C_1+B+B')h_j = Ah_j$ uniformly in $0 \leq t \leq T$, where clearly $C_1,B,B'$ depend only on $a$ and $T$. Now, this yields $|s_j(t)-s_j(0)| \leq ATh_j$ when integrating.
\end{proof}

A direct consequence is the following:
\begin{cor}
Let $T > 0$, let $a \in C^\infty_c(T^*X)$ with $\supp(a)\subset \Sigma_T$. There exists some constant $C > 0$ such that for every $j \geq 0$, \[\left|\langle \Op_h(a-\langle a \rangle_T)u_j,\,u_j\right| \leq C h_j\]
with $\langle a \rangle_T:=\frac{1}{T}\int_{0}^T a\circ \Phi^t dt$.
\end{cor}

An easy argument then yields (taking into account the fact that $a-\langle a \rangle_T$ has average zero): 

\begin{prop}
\label{egorov-goal}
Let $T > 0$, let $a \in C^\infty_c(T^*X)$ with $\supp(a)\subset \Sigma_T$. Then, as $\lambda \longrightarrow \infty$, $M(a-\langle a \rangle_T,\lambda) \longrightarrow 0$. 
\end{prop}

\section{Analysis far in the cusp}

If we joined the main results of sections $3$ and $4$, we would be able to prove the main theorem for symbols with average zero. This will be done in section $6$.  

But if we want to prove the main theorem for general symbols in $S_{\rm comp}^0$, 
we have to find some symbols with non-zero average for which the result holds. A direct proof turns out to be difficult: so we will exhibit symbols $s$ with average arbitrarily close to some non-zero constant and such that $\limsup_{\lambda \rightarrow \infty}\,M(s,\lambda)$ is arbitrarily close to $0$.
\\

Before, we need to introduce some cutoff functions:
let us set some $R > e^{c_0}$, and $\chi_R$ be a smooth nondecreasing function such that 
\begin{equation}\label{chiR}
\chi_R(r)=1 \textrm{ on }[R+1,\infty), \quad \chi_R(r)=0  \,\,  \textrm{ on } (-\infty,R].
\end{equation}
Let $\phi_R:X \rightarrow [0;1]$ be a smooth function that is zero outside the cusp and such that if $r > c_0,\theta \in \R/\Z$, $\phi_R (r,\theta) = \chi_R(r)$. Note that $1-\phi_R^4\in C_c^\infty(X)$.

We will show the following: 
\begin{prop}
\label{cuspanalysis-goal}
There exists a universal constant $C > 0$ such that for any $R > e^c$,
$$\limsup_{\lambda \rightarrow \infty}\,M(1-\phi_R^{8},\lambda) \leq Ce^{-R/4}.$$
\end{prop}
Our first step is to understand where the mass of the $u_j$ is localized.

Let us write, for every $j \geq 0, r > c_0, \theta \in \R/\Z$, $$u_j(r,\theta) = e^{r/2}\sum_{k\in \Z}{v_{j,k}(r)e^{2ik\pi\theta}}.$$ This is similar to the expansion of $u_j$ as a Fourier series in $\theta$, but the coefficients were renormalized, so that $\int_{\theta \in \R/\Z}{|u_j|^2\,d{\rm v}_g} = \sum_{k \in \Z}{|v_{j,k}|^2(r)}\,dr$.  

Let $\chi_1$ be a smooth nondecreasing function such that $\chi_1=0$ on $(-\infty,0]$, and $\chi_1=1$ on $[1,\infty)$. 
We will now denote $\chi_{h,k}(r) = \chi_1(-\ln{2|k|h\pi}-r+R/2)$. This function is going to be used to weaken the growth of the function $(2\pi kh)^2e^{2r}$, that is not a symbol: indeed, $\chi_{h,k}(r) = 0$ as soon as $(2\pi hk)^2e^{2r} \geq e^R$. 

Let $\chi_2$ be a smooth non-negative compactly supported function such that $\chi_2=1$ on $[-3,3]$, and $\chi_2=0$ outside $(-4,4)$, and $|\chi_2| \leq 1$.

We first have the following straightforward formula: for $u\in H^1(X)$, we have in the cusp, 
\[\,|\nabla^g u|^2_g(r,\theta) = |(\partial_r u)(r,\theta)|^2+ e^{2r}|(\partial_{\theta}u)(r,\theta)|^2.\]

\begin{cor}
\label{L2_v_bounds}
For $j \geq 0$, the following bounds hold true: 
\begin{align}
\label{exp_rmd_est}
\sum_{k \in \Z}{(2\pi k)^2\int_{c_0}^{\infty}{e^{2r}|v_{j,k}|^2(r)\,dr}} &\leq \lambda_j^2 + \frac{1}{4},\\
\label{norm_est}
\sum_{k \in \Z}{\int_{c_0}^{\infty}{|v_{j,k}|^2(r)\,dr}} &\leq 1,\\
\label{dvt_norm_est}
\sum_{k \in \Z}{\int_{c_0}^{\infty}{|v'_{j,k}|^2(r)\,dr}} &\leq 2\lambda_j^2 + 1
\end{align}
\end{cor}

\begin{proof}
For the bound \eqref{exp_rmd_est}, we estimate
\begin{align*}
\lambda_j^2 + \frac{1}{4}&= \langle u_j,\,\Delta_c u_j\rangle = \|\nabla u_j\|^2_{L^2} \geq \int_{r > c_0}|\nabla^g u_j|^2_g\,d{\rm v}_g\\
&\geq \int_{r > c_0}{\int_{\R/\Z}{e^{2r}|\partial_{\theta}u_j(r,\theta)|^2d\theta}e^{-r}dr}\\
&\geq \sum_{k \in \Z}{(2k\pi)^2\int_{r > c_0}{e^{2r}|v_{j,k}|^2(r)\,dr}}.
\end{align*}
The bound \eqref{norm_est} is a direct consequence of the fact that $\|u_j\|^2_{L^2} = 1$. 
As for \eqref{dvt_norm_est}, note that in the cusp, $(\partial_r u_j)(r,\theta) = e^{r/2}\sum_{k \in \Z}{\left(\frac{1}{2}v_{j,k}+v'_{j,k}\right)e^{2ik\pi\theta}}$, so
\[\begin{split}
\lambda_j^2 + \frac{1}{4} & \geq \sum_{k \in \Z}{\int_{c_0}^{\infty}{\int_{\R/\Z}{e^{-r}|\partial_r u_j(r,\theta)|^2\,d\theta}dr}}  \geq \sum_{k \in \Z}{\int_{c_0}^{\infty}{\left|\frac{v_{j,k}}{2}+v'_{j,k}\right|^2\,dr}} \\
&\geq \sum_{k \in \Z}{\int_{c_0}^{\infty}{\left(\frac{1}{2}|v'_{j,k}|^2-\left|\frac{v_{j,k}}{2}\right|^2\right)\,dr}} \geq \frac{1}{2}\sum_{k \in \Z}{\int_{c_0}^{\infty}{|v'_{j,k}|^2(r)\,dr}} - \frac{1}{4}
\end{split}\]
and the proof is complete.\end{proof}

Let now $h > 0$ be very small, and $\lambda = h^{-1}$. Let $\omega_j^2 = (h\lambda_j)^2$, for every $j$ such that $h/2 \leq h_j \leq 2h$ (we say that $j$ is \emph{in the range}): then, if $h$ is small enough, $\omega_j$ is between $1/2$ and $2$. 

\begin{lem}
In the cusp region $r > c_0$, 
\[(-h^2\partial_r^2 - \omega_j^2 + (2kh\pi)^2e^{2r})v_{j,k} = 0.\] 
We can therefore extend $v_{j,k}$ as a smooth function of the whole real line with the same properties.  
\end{lem}

This means that microlocally, the mass of $v_{j,k}$ is concentrated near the curve $\xi^2+(2kh\pi)^2e^{2r} = \omega_j^2$:

\begin{lem}
There exists a constant $C_R$ depending only on $R$ (and not on $h$) such that for every $j$ in the range, for every $1 \leq |k| \leq 3\lambda$, $$\|\Op_h((1-\chi_2)(\xi)\chi_R^2(r)\chi_{h,k}(r)^2)(\chi_Rv_{j,k})\|_{L^2}^2 \leq C_Rh^2\|v_{j,k}\|^2_{L^2(c_0,\infty)}+C_Rh^4\|v'_{j,k}\|_{L^2(c_0,\infty)}^2.$$ 
\end{lem}

\begin{proof}
Let us denote here
\begin{align*}
f_{j,k}^{(1)} &:= \Op_h\left(\chi_R(r)\chi_{h,k}(r)(\xi^2-\omega_j^2+(2kh\pi)^2e^{2r})\right)(\chi_Rv_{j,k}) \\
&= \chi_R(r)\chi_{h,k}(r)(-2h^2\chi_R'(r)v_{j,k}'-h^2\chi_R''(r)v_{j,k}),
\end{align*} and 
\[f_{j,k} := \Op_h\left(\chi_R(r)\chi_{h,k}(r)\frac{1-\chi_2(\xi)}{\xi^2-\omega_j^2+(2kh\pi)^2e^{2r}}\right)f_{j,k}^{(1)}.\]
Now, using \cite[Theorem 4.23]{EZ}, (remember that inside the symbol, when $(2|k|h\pi)e^r>C_R$, the $\chi_1$ term destroys everything, so all relevant $S^0$ seminorms are bounded uniformly in $j$, $k$, $h$, as long as $h\leq |k|h \leq 3$ and $j$ is in the range) we may write, for some constant $K_R > 0$ depending only on $R$:
\begin{align*}
\left\|\frac{1}{h^2}f_{j,k}\right\|_{L^2}^2 &= \frac{1}{h^4}\left\|\Op_h\left(\chi_R(r)\chi_{h,k}(r)\frac{1-\chi_2(\xi)}{\xi^2-\omega_j^2+(2hk\pi)^2e^{2r}}\right)f^{(1)}_{j,k}\right\|_{L^2}^2\\
&\leq \frac{1}{h^4}\left\|\Op_h\left(\chi_R(r)\chi_{h,k}(r)\frac{1-\chi_2(\xi)}{\xi^2-\omega_j^2+(2hk\pi)^2e^{2r}}\right)\right\|_{L^2 \rightarrow L^2}^2\|f_{j,k}^{(1)}\|^2\\
&\leq K_R(\|v_{j,k}\|^2_{L^2} + \|v'_{j,k}\|^2_{L^2}).
\end{align*}
Now, when $h$ is small, $1 \leq |k| \leq 3\lambda$, $j$ is in the range, the symbols 
\begin{align*}
a_- := \chi_R(r)\chi_{h,k}(r)\frac{1-\chi_2(\xi)}{\xi^2-\omega_j^2+(2kh\pi)^2e^{2r}}&, \quad a_+ := \chi_R(r)\chi_{k,h}(r)(\xi^2-\omega_j^2+(2kh\pi)^2e^{2r});\\
&a_* := a_-a_+
\end{align*}
are bounded by a constant depending only on $R$ in respectively the class of symbols $S(\langle \xi \rangle^{-2})$, $S^0(\langle \xi \rangle^2)$, $S(1)$ (using the notation of \cite[Section 4.4.1]{EZ}), thus by \cite[Theorems 4.18, 4.23]{EZ} we have $\|\Op_h(a_-)\Op_h(a_+)-\Op_h(a_*)\|_{L^2 \rightarrow L^2} \leq C_Rh$, for some constant $C_R > 0$ depending only on $R$. 
Therefore, we get that for some constant $C_R > 0$ depending only on $R$, 
\begin{align*}
\|\Op_h(a_*)(\chi_Rv_{j,k})\|^2_{L^2} &\leq 2C_Rh^2\|\chi_Rv_{j,k}\|^2_{L^2} + 2\|f_{j,k}\|^2_{L^2}\\
&\leq C_R(h^2\|v_{j,k}\|^2_{L^2(c_0,+\infty)} + h^4\|v'_{j,k}\|^2_{L^2(c_0,+\infty)}).
\end{align*}
\end{proof}

\begin{cor}
\label{phase_loc}
There exists a constant $C_R$ depending only on $R$ such that when $j$ is in the range 
$$\sum_{1 \leq |k| \leq 3\lambda}{\|\chi_R(r)\Op_h((1-\chi_2)(\xi)\chi_R^2(r)\chi_{h,k}^2(r))(\chi_Rv_{j,k})\|_{L^2}^2} \leq C_Rh^2.$$ 
\end{cor}

\begin{proof}
It is a consequence of the previous lemma and of corollary \ref{L2_v_bounds}, more precisely estimates (\ref{norm_est}) and (\ref{dvt_norm_est}). 
\end{proof}
The result hereafter is the main property of localization we were aiming at: it tells us that each $\phi_R^4u_j$ is localized along his lowest modes in the cusp, and each of these modes is microlocalized in a compact zone that depends very little on $j$: that is, $v_{j,k}$ is microlocalized in the zone $|\xi| \leq 4$, $R \leq r \leq R/2-\ln{2h|k|\pi}$.

Let us first define the operator $A_{h,k} := \Op_h(\chi_2(\xi)\chi_R^2(r)\chi_{h,k}(r)^2)$. 
\begin{prop}
\label{full_loc}
Let $j \geq 0$ be in the range. For some universal constant $C > 0$, and some constant $C_R > 0$ depending only on $R$, 
$$\|\phi_R^4u_j\|_{L^2}^2 \leq 4\sum_{1 \leq |k| \leq 3\lambda}{\|\chi_R(r)A_{h,k}(\chi_Rv_{j,k})\|^2_{L^2}} + Ce^{-R} + C_Rh^2.$$
\end{prop}

We split the proof in several steps.

\begin{lem}
\label{mode_loc}
Let $j \geq 0$ be in the range. Then 
$$\sum_{|k| > 3\lambda}{\|\chi_R^4v_{j,k}\|_{L^2}^2} \leq e^{-2R}.$$
\end{lem}

\begin{proof}
Using (\ref{exp_rmd_est}), we obtain the sequence of inequalities
\[\begin{split}
\sum_{|k| > 3\lambda}{\|\chi_R^4v_{j,k}}\|_{L^2}^2 &\leq \frac{1}{36}\sum_{|k| > 3\lambda}{e^{-2R}\frac{(2k\pi)^2}{\lambda^2}\int_{R}^{\infty}{e^{2r}|v_{j,k}|^2\,dr}} \\
& \leq \frac{1}{36e^{2R}\lambda^2}\sum_{k \in \Z}{(2\pi k)^2\int_{c_0}^{\infty}{e^{2r}|v_{j,k}|^2\,dr}} \leq e^{-2R}\frac{\lambda_j^2+\frac{1}{4}}{36\lambda^2} \leq e^{-2R}
\end{split}\]
which proves the claim.
\end{proof}

\begin{lem}
\label{space_loc}
Let $j \geq 0$ be in the range. Then the following holds true:
$$\sum_{1 \leq |k| \leq 3\lambda}{\|\chi_R^4(1-\chi_{h,k}^2)v_{j,k}\|_{L^2}^2} \leq 40e^{-R} + 3h^2$$
\end{lem}

\begin{proof}
Using again (\ref{exp_rmd_est}),
\begin{align*}
\sum_{1 \leq |k| \leq 3\lambda}{\|\chi_R^4(1-\chi_{h,k}^2)v_{j,k}\|_{L^2}^2} &\leq \sum_{1 \leq |k| \leq 3\lambda}{\int_{R/2-1-\ln{2\pi|k|h}}^{\infty}{|v_{j,k}|^2(r)\,dr}}\\
&\leq \sum_{1 \leq |k| \leq 3\lambda}{\int_{R/2-1-\ln{2\pi|k|h}}^{\infty}{e^{-R}(2e\pi kh)^2e^{2r}|v_{j,k}|^2(r)\,dr}}\\
&\leq e^2h^2e^{-R}\sum_{k \in \Z}(2\pi kh)^2{\int_{c_0}^{\infty}{e^{2r}|v_{j,k}|^2(r)\,dr}}\\
&\leq 40e^{-R}+\frac{1}{4}e^{2-R}h^2 \leq 40e^{-R}+3h^2
\end{align*}
which proves the claim.
\end{proof}

Now, we can prove Proposition \ref{full_loc}:

\begin{proof}[Proof of Proposition \ref{full_loc}]
One easily sees that:
$$\|\phi_R^4u_j\|_{L^2}^2 = \sum_{|k| \geq 1}{\int_R^{\infty}{\chi_R(r)^8|v_{j,k}|^2(r)\,dr}}\leq \sum_{ |k|\geq 1}{\|\chi_R^4v_{j,k}\|_{L^2}^2}.$$
Now, we split the sum between the $|k| > 3\lambda$, the sum of which is not greater than $e^{-2R}$ (Lemma \ref{mode_loc}), and the $1 \leq |k| \leq 3\lambda$.
Moreover, if $1 \leq |k| \leq 3\lambda$, 
$$\|\chi_R^4v_{j,k}\|^2_{L^2} \leq 2\|\chi_R^4(1-\chi_{h,k}^2)v_{j,k}\|^2_{L^2}+2\|\chi_R^4\chi_{h,k}^2v_{j,k}\|^2_{L^2}.$$
From Lemma \ref{space_loc} we see that the first term contributes only as $6h^2+80e^{-R}$, now we have to assess the second term. 
Now, $$2\|\chi_R^4\chi_{h,k}^2v_{j,k}\|^2_{L^2} \leq 4\|\chi_RA_{h,k}(\chi_Rv_{j,k})\|^2_{L^2} + 4\|\chi_R\Op_h((1-\chi_2(\xi))\chi_{h,k}^2(r)\chi_R(r)^2)(\chi_Rv_{j,k})\|^2_{L^2}.$$
We saw from Corollary \ref{phase_loc} that the second term contributes only as $C_Rh^2$, where $C_R > 0$ depends only on $R$, and this ends the proof. 
\end{proof}

Now, we turn the pointwise localization estimate we have on the $\phi_R^4u_j$ into an average estimate on $j$. Thanks to Hilbert-Schmidt norm estimates (operators will always be considered as from the relevant $L^2$ spaces into themselves) we obtain significantly better results.

Let us define the operator $A'_{h,k} := A_{h,k}\chi_R$, then let 
\[Aw(r,\theta) := \chi_R(r)e^{r/2}\sum_{1 \leq |k| \leq 3\lambda}{(A_{h,k}(\chi_Rw_k))(r)e^{2ik\pi\theta}}\] 
for every $r > c_0,\theta \in \R/\Z$, where $w(r,\theta) = e^{r/2}\sum_{k \in \Z}{w_k(r)e^{2ik\pi\theta}}$.

\begin{prop}
There exist constants $C > 0$ universal, and $C_R > 0$ depending on $R$ only such that:
\[h^2\sum_{h/2 \leq h_j \leq 2h}{\|\phi_R^4u_j\|^2_{L^2}} \leq C_Rh^2 + Ce^{-R} + 4h^2\sum_{1 \leq |k| \leq 3\lambda}{\|A'_{k,h}\|^2_{\mathrm{HS}}}.\]
\end{prop}

\begin{proof}
From Proposition \ref{full_loc}, we know that for any $j$ in range, $\|\phi_R^4u_j\|^2 \leq 200e^{-R} + C_Rh^2 + 4\|Au_j\|^2_{L^2}$.
So, using Weyl's law, and the fact the the $(u_j)$ are orthonormal, $$h^2\sum_{h/2 \leq h_j \leq 2h}{\|\phi_R^4u_j\|^2_{L^2}} \leq Ce^{-R} + C_Rh^2 + 4h^2\|A\|^2_{\mathrm{HS}}.$$
Now, let $(f_p)$ be an orthonormal basis of $L^2(R,\infty)$, let $(g_q)$ be an orthonormal basis of $\{f \in L^2(X),\,\mathbbm{1}_{r > R}f = 0\}$. Then the family of all $e^{r/2}f_p(r)e^{2ik\pi\theta}$ and the $g_q$, is an orthonormal basis of $L^2(X)$. 
Therefore, when $f$ is an element of this orthonormal basis, we realize that only when $f = e^{r/2}f_p(r)e^{2ik\pi\theta}$, $1 \leq |k| \leq 3\lambda$, $\|Af\|^2$ does not vanish. Besides, it will always be lower or equal than $\|(\chi_R \circ A_{h,k}\circ\chi_R)f_p\|^2_{L^2(\R)}$. From this, it follows that 
$$\|A\|^2_{\mathrm{HS}} \leq \sum_{1 \leq |k| \leq 3\lambda}{\|\chi_R \circ A_{k,h} \circ \chi_R\|^2_{\mathrm{HS}}} \leq \sum_{1 \leq |k| \leq 3\lambda}{\|A'_{h,k}\|^2_{\mathrm{HS}}},$$
which completes the proof.
\end{proof}

Now, it is easy to give an upper bound on the Hilbert-Schmidt norm of the operators, and to turn it into a complete estimate:

\begin{prop} 
\label{bound on A'hk}
The following bound holds true for $|k|\leq 3\lambda$:
$$\|A'_{h,k}\|^2_{\mathrm{HS}} \leq 4(h\pi)^{-1}\left(-\ln{2|k|h\pi}-R/2\right)^+$$
\end{prop}

\begin{proof}
Let $\psi \in \CIc(\R^2)$, $\hat{\psi}$ be its Fourier transform with respect to its second variable. Let $T = \Op_h(\psi)\chi_R$. For any $f \in L^2(\R)$, 
$$2h\pi Tf(x) = \int_{\R^2}{e^{i\frac{\xi(x-y)}{h}}\psi\left(\frac{x+y}{2},\xi\right)\chi_R(y)f(y)\,dyd\xi} = \int_{\R}{\hat{\psi}\left(\frac{x+y}{2},\frac{y-x}{h}\right)\chi_R(y)f(y)\,dy},$$
therefore 
\[\begin{split}
\|2h\pi T\|^2_{{\rm HS}} = & \int_{\R^2}{\left|\hat{\psi}\left(\frac{x+y}{2},\frac{y-x}{h}\right)\right|^2\chi_R^2(y)\,dxdy} \leq \int_{\R^2}{\left|\hat{\psi}\left(\frac{x+y}{2},\frac{y-x}{h}\right)\right|^2\,dxdy}\\
&\leq h\int_{\R^2}{\left|\hat{\psi}(x,y)\right|^2\,dydx} = 2h\pi\|\psi\|_{L^2}^2.
\end{split}\]
Now, we obtain
\[\begin{split}
\|\chi_2(\xi)\chi_R^2(r)\chi_1^2(-\ln{2|k|h\pi}-r+R/2)\|^2_{L^2} &
\leq \int_{\R^2}\mathbbm{1}_{\{|\xi| \leq 4\}}\mathbbm{1}_{\{r \geq R\}}\mathbbm{1}_{\{r \leq R/2
-\ln{2|k|h\pi}\}}\,d\xi dr \\
& \leq 8\left(-\ln{2|k|h\pi}-R/2\right)^+.
\end{split}\]
This completes the proof.
\end{proof}

\begin{cor}
The following estimate holds true:
$$4h^2\sum_{1 \leq |k| \leq 3\lambda}{\|A'_{k,h}\|^2_{\mathrm{HS}}} \leq Ce^{-R/2} + C_Rh,$$ where $C> 0$ is universal and $C_R$ depends only on $R$.
\end{cor}

\begin{proof}
Using Stirling's formula, assuming that $2\pi he^{R/2} \leq 1$ (else anyway the sum is just zero and the bound holds), for some universal constant $C$,
\begin{align*}
\frac{h}{4}\sum_{1 \leq |k| \leq 3\lambda}{\|A'_{k,h}\|^2_{HS}} &\leq \sum_{1 \leq k \leq 3\lambda}{\left(-\ln{2kh\pi}-\frac{R}{2}\right)^+}\\
&\leq \sum_{1 \leq k \leq (2\pi)^{-1}\lambda e^{-R/2}}{-\ln{2kh\pi}-R/2}\\
&\leq -\frac{\lambda e^{-R/2}}{2\pi}\ln{2h\pi} - \frac{\lambda e^{-R/2}}{2\pi}\frac{R}{2} - \sum_{1 \leq k \leq (2\pi)^{-1}\lambda e^{-R/2}}{\ln{k}}\\
&\leq \frac{\lambda e^{-R/2}}{2\pi}\ln{\frac{\lambda e^{-R/2}}{2\pi}} - \frac{\lambda e^{-R/2}}{2\pi}\ln\left(\frac{\lambda e^{-R/2}}{2\pi}\right) + \frac{\lambda e^{-R/2}}{2\pi} + C\\
&\leq \frac{e^{-R/2}}{2h\pi} +  \frac{Ch}{h}
\end{align*}
\end{proof}

\begin{cor}\label{boundY}
For some universal constant $C > 0$ and some constant $C_R$ depending only on $R$, 
one has $Y(\phi_R^4,h)^2 \leq Ce^{-R/2} + C_Rh$.
\end{cor}
\begin{proof}
It is a consequence of all that precedes.
\end{proof}

\begin{proof}[Proof of Proposition \ref{cuspanalysis-goal}] 
We write 
\begin{align*}
M(1-\phi_R^8, \lambda)^2 &= \frac{1}{N(\lambda)}\sum_{\lambda_j \leq \lambda}{\left|\langle\Op_{h_j}(1-\phi_R^8)u_j,\,u_j\rangle-\int_{S^*X}{1-\phi_R^8}\right|^2}\\
&\leq \frac{1}{N(\lambda)}\sum_{\lambda_j \leq \lambda}
\left|\int_X \phi_R^8|u_j|^2\,d{\rm v}_g -\int_{S^*X}{\phi_R^8}\right|^2\\
&\leq \frac{1}{N(\lambda)}\sum_{\lambda_j \leq \lambda}\left|\int_X\phi_R^8|u_j|^2\,d{\rm v}_g\right|^2 + \left(\int_{S^*X}{\phi_R^8}\right)^2\\
&= \frac{1}{N(\lambda)}\sum_{\lambda_j \leq \lambda}\left|\int_X\phi_R^8|u_j|^2\,d{\rm v}_g\right|^2 + \frac{1}{{\rm Vol}(X)}\left(\int_X \phi_R^8\,d{\rm v}_g\right)^2.
\end{align*}

The second term is lower than some $Ce^{-2R}$ for some constant $C$ independent of $R$, because $d{\rm v}_g(r,\theta) = e^{-r}drd\theta$.
The first term is not greater than 
$$\frac{1}{N(\lambda)}\sum_{1 \leq 4^k \leq 4h_0\lambda}{2^{2-4k}\lambda^2Y\left(\phi_R^4,\frac{2^{2k-1}}{\lambda}\right)^2},$$
which by Corollary \ref{boundY} is not greater than (using again Weyl's law)
\[\frac{C}{\lambda^2}\sum_{1 \leq 4^k \leq 4h_0\lambda}{Ce^{-R/2}2^{2-4k}\lambda^2 + C_R2^{1-2k}\lambda} \leq Ce^{-R/2} + C_R\frac{1}{\lambda}.\]
This concludes the proof. \end{proof}

\section{Proof of the main theorem}

Let $a \in S_{\rm comp}^0$ and assume first that $\int_{S^* X}{a\,d\mu} = 0$.
We let $T > 0$ and $\eps > 0$. We may write $a = a_1+a_2$, where $a_1$ is $S^{-\infty}$ 
satisfies $\supp(a_1)\subset \Sigma_T$, and $\int_{S^*X}{|a_2|^2} \leq \eps^2$.
Then, as $\lambda \rightarrow \infty$,  
\[M(a,\lambda) \leq M(a_1,\lambda) + C\eps \leq M(\langle a_1\rangle_T,\lambda) + C\eps \leq M(\langle a\rangle_T,\lambda) + 2C\eps \leq C(\|\langle a\rangle_T\|_{L^2(S^*X)} + 2\eps)\] 
where we used Proposition \ref{variancebound} for the last inequality.
So we let $\eps$ go to zero first, and then $T$ go to $\infty$ and the $L^2$ ergodic theorem proves the result.

In the general case, let $a \in S_{\rm comp}^0(X)$ and let $\alpha:=\int_{S^*X}a$. Let us denote, for every $R > e^{c_0}$, $I_R:=\int_{S^*X}\phi_R^8$. Then $I_R = \mathcal{O}(e^{-R})$, thus, if $R$ is large enough, $a_R = a-\frac{\alpha}{1-I_R}(1-\phi_R^8)$ belongs to  $S_{\rm comp}^0$ with 
\[\int_{S^*X}a_R=0.\]
Thus we have $M(a_R,\lambda) \rightarrow 0$ as $\lambda\to \infty$. Now, by Proposition \ref{cuspanalysis-goal} 
\[\limsup_{\lambda \rightarrow \infty}\,M\left(\frac{\alpha}{1-I_R}(1-\phi_R^8),\lambda\right) \leq \frac{C|\alpha|}{1-I_R}e^{-R/4}.\] 
Thus, \[\limsup_{\lambda \rightarrow \infty}\,M(a,\lambda) \leq \frac{C|\alpha|}{1-I_R}e^{-R/4}.\] 
Letting $R \rightarrow \infty$ yields $M(a,\lambda) \rightarrow 0$ and the proof of Theorem \ref{main} is complete.

\appendix

\section{Trace of a pseudo-differential operator}

\begin{prop}
\label{trace}
Let $a \in C^\infty_c(T^*X)$. Then ${\rm Op}_h(a)$ is trace class and
\[\Tr_{L^2 \rightarrow L^2}\left(\Op_h(a)\right) = \frac{1}{(2\pi h)^2}\int_{T^*X}{a(x,\xi)\,dxd\xi}.\]
\end{prop}
\begin{proof}
We may assume that $a$ is real-valued.
Let $u \in \CIc(X)$. Let $x \in X$, let $i \geq 0$ be such that $x \in U_i$. Using a change of variables, the following holds:
\begin{align*}
&4h^2\pi^2(\chi_i(\varphi_i)^*((\varphi_i)_*a^w(x,hD))(\varphi_i)_*\chi_iu)(x) \\
&=\chi_i(x)\int_{U_i}{\int_{T^*_yX}{a(m_i(x,y),(\varphi_i)_{y \rightarrow m_i(x,y)}(\xi))\chi_i(y)u(y)\exp\left(\frac{i}{h}\left(\varphi_i(x)-\varphi_i(y)\right)\cdot(\varphi_i)_{*,y}(\xi)\right)\,d\xi}\,dy}\\
&=\int_{X}{f_i(x,y)u(y)\,dy},
\end{align*}

where $$m_i(x,y) = \varphi_i^{-1}\left(\frac{\varphi_i(x)+\varphi_i(y)}{2}\right),$$
$(\varphi_i)_{*,y}$ is the linear mapping from $T^*_yX$ into $\R^2$ such that the thus enhanced $\varphi_i$ is a symplectomorphism, 
$(\varphi_i)_{y \rightarrow m_i(x,y)} = (\varphi_i)_{*,m_i(x,y)}^{-1} \circ (\varphi_i)_{*,y}$. 
$$f_i(x,y) = \chi_i(x)\chi_i(y)\int_{T^*_yX}{a(m_i(x,y),(\varphi_i)_{y \rightarrow m_i(x,y)}(\xi))\exp\left(\frac{i}{h}\left(\varphi_i(x)-\varphi_i(y)\right)\cdot(\varphi_i)_{*,y}(\xi)\right)\,d\xi},$$ the measure $d\xi$ being the Liouville measure. The integral is convergent because $a$ has compact support. 

Let $F$ be the set of all indices $i$ such that the space support of $a$ meets $U_i$; $F$ is finite. For every $i \notin F$, $f_i=0$; for every $i \in F$, $f_i:(\supp\,\chi_i)^2 \rightarrow \C$ is continuous, hence bounded; thus, $F = \sum_i{f_i}$ is in $L^{\infty}(X^2)$.

Thus, the operator
$$v \in L^2(X) \longmapsto 4h^2\pi^2\Op_h(a)v - \left(x \longmapsto \int_X{F(x,y)v(y)\,dy}\right) \in L^2(X)$$ is bounded and vanishes at $u$, so vanishes on $\CIc(X)$, and we have 
$$\forall u \in L^2(X),\, 4h^2\pi^2[\Op_h(a)u](x) = \int_X{F(x,y)u(y)\,dy},$$ hence 
\begin{align*}
\Tr\left(\Op_h(a)\right) &= \frac{1}{4h^2\pi^2}\int_X{F(x,x)\,dx} = \frac{1}{4h^2\pi^2}\sum_i{\int_X{f_i(x,x)\,dx}} \\
&= \frac{1}{4h^2\pi^2}\sum_i{\int_X{\chi_i^2(x)\int_{T^*_xX}{a(x,\xi)\,d\xi}\,dx}}\\
&= \frac{1}{4h^2\pi^2}\int_{T^*X}{a(x,\xi)\,dxd\xi}.
\end{align*}
\end{proof}

\end{document}